\documentclass{article}
\usepackage{amsmath,amsthm,amssymb,amscd}

\theoremstyle{definition}
\newtheorem{Thm}{Theorem}[section]

\newtheorem{Lem}[Thm]{Lemma}
\newtheorem{Cor}[Thm]{Corollary}

\newtheorem{Rem}[Thm]{Remark}

\def\H{\mathfrak{H}}

\title{On Cannon-Thurston maps for relatively hyperbolic groups}
\date{}
\author{
Yoshifumi Matsuda\footnote
{Graduate School of Mathematical Sciences, 
University of Tokyo, 
3-8-1 Komaba, 
Meguro-ku, 
Tokyo, 
153-8914 Japan,
ymatsuda@ms.u-tokyo.ac.jp} \footnote{
The author is supported by the Global COE Program at Graduate School of Mathematical Sciences, the University of Tokyo, and Grant-in-Aid for Scientific Researches for Young Scientists (B) (No. 22740034), Japan Society of Promotion of Science.}
and 
Shin-ichi Oguni\footnote
{Department of Mathematics, Faculty of Science, Ehime University,
2-5 Bunkyo-cho, 
Matsuyama, 
Ehime, 
790-8577 Japan,
oguni@math.sci.ehime-u.ac.jp}
\footnote{
The author is supported by Grant-in-Aid for Scientific Researches for Young Scientists (B) 
(No. 24740045), Japan Society of Promotion of Science.}}
\begin{document}

\maketitle
%%%%%%%%%%%%%%%%%%%%%%%%%%%%%%%%%%%%%%%%%%%%%%%%%
\begin{abstract}
Baker and Riley proved that a free group of rank $3$
can be contained in a hyperbolic group as a subgroup 
for which the Cannon-Thurston map is not well-defined. 
By using their result, 
we show that the phenomenon 
occurs for not only a free group of rank $3$
but also every non-elementary hyperbolic group. 
In fact it is shown that a similar phenomenon occurs for 
every non-elementary relatively hyperbolic group. \\

\noindent
Keywords:
Cannon-Thurston maps; 
relatively hyperbolic groups;
geometrically finite convergence actions;  
convergence actions \\

\noindent
2010MSC:
20F67;
20F65
\end{abstract}

\section{Introduction}
Given an injective group homomorphism from a hyperbolic group to another hyperbolic group, 
whether the map can be continuously extended on the Gromov boundaries 
is an interesting question by Mitra (see \cite[Section 1]{Mit98b}). 
If such an extension is well-defined, 
the induced map on the Gromov boundaries is called the Cannon-Thurston map. 
The first non-trivial example was known by Cannon and Thurston in the 1980's 
(see \cite{C-T07}). 
Indeed their main theorem implies that 
for a closed hyperbolic $3$-dimensional manifold $M$ which fibers over the circle
with fiber a closed hyperbolic surface $S$, 
when we consider 
the induced injective group homomorphism between fundamental groups of $S$ and $M$, 
the Cannon-Thurston map is well-defined. 
Also more examples for which 
the Cannon-Thurston maps are well-defined can be recognized by 
Mitra's results (see \cite{Mit98a} and \cite{Mit98b}). 
At the present time, there are many works 
related to well-definedness of the Cannon-Thurston maps.
Nevertheless Baker and Riley gave a negative answer (\cite[Theorem 1]{B-R12}). 
Indeed they showed that a free group of rank $3$
can be contained in a hyperbolic group as a subgroup 
for which the Cannon-Thurston map is not well-defined.
In this paper we show that the phenomenon 
occurs for not only a free group of rank $3$
but also every non-elementary hyperbolic group. 
In fact it is shown that a similar phenomenon occurs for 
every non-elementary relatively hyperbolic group.  

Throughout this paper, 
every countable group is endowed with the discrete topology.
We use a definition of relative hyperbolicity for groups 
from a dynamical viewpoint 
(see \cite[Definition 1]{Bow12}, \cite[Theorem 0.1]{Yam04} and \cite[Definition 3.1]{Hru10}). 
Also we use a definition of relative quasicovexity for subgroups of 
relatively hyperbolic groups from a dynamical viewpoint
(see \cite[Definition 1.6]{Dah03}). 
Refer to \cite[Section 3 and Section 6]{Hru10} 
for other several equivalent definitions of those.
Also see \cite{Tuk94}, \cite{Bow99b} and \cite{Bow12} 
for some definitions and properties related to convergence actions. 

Let $G$ be a non-elementary countable group 
and $\H$ be a conjugacy invariant collection of proper infinite subgroups of $G$. 
Suppose that $G$ is hyperbolic relative to $\H$, that is, 
there exists a compact metrizable space 
endowed with a geometrically finite convergence action of $G$
such that $\H$ is the set of all maximal parabolic subgroups of $G$. 
Such a space is unique up to $G$-equivariant homeomorphisms
and called the Bowditch boundary of $(G,\H)$. 
In this paper we denote it by $\partial (G,\H)$. 
We remark that the set of conjugacy classes of elements of $\H$ is automatically finite
by \cite[Theorem 1B]{Tuk98}.
When the group $G$ is hyperbolic, 
it is hyperbolic relative to the empty collection $\emptyset$
and the Bowditch boundary $\partial (G,\emptyset)$
is nothing but the Gromov boundary $\partial G$.
 
We consider another non-elementary countable group $G'$ 
which is hyperbolic relative to a conjugacy invariant collection $\H'$
of proper infinite subgroups of $G'$. 
Suppose that $G$ is a subgroup of $G'$. 
Then we can consider the restricted action of $G$ on $\partial(G',\H')$
and the limit set $\Lambda(G,\partial(G',\H'))$. 
If there exists a $G$-equivariant continuous map 
from $\partial(G, \H)$ to $\partial(G',\H')$, 
then it is unique and the image is equal to $\Lambda(G,\partial(G',\H'))$
(see for example \cite[Lemma 2.3 (1), (2)]{M-O-Y5}). 
When the map exists, it is also called the Cannon-Thurston map. 
If the Cannon-Thurston map is well-defined, 
then any $H\in \H$ is contained in some $H'\in \H'$
(see for example \cite[Lemma 2.3 (5)]{M-O-Y5}). 
In general the converse 
is not true by \cite[Theorem 1]{B-R12} (see also Lemma \ref{extension}). 
Our main theorem claims that the converse is also not true for 
the case where the pair of $G$ and $\H$ 
is not necessarily the pair of a free group of rank $3$ and $\emptyset$. 
More precisely we have the following: 
\begin{Thm}\label{noCT}
Let $G$ be a non-elementary countable group 
which is hyperbolic relative to a conjugacy invariant collection $\H$
of proper infinite subgroups of $G$. 
Then there exist a countable group $G'$ containing $G$ as a subgroup 
and a conjugacy invariant collection $\H'$ of proper infinite subgroups of $G'$
satisfying the following: 
\begin{enumerate}
\item[(i)] the group $G'$ is hyperbolic relative to $\H'$;
\item[(ii)] every $H\in\H$ belongs to $\H'$ and 
each $H'\in \H'$ is conjugate to some $H\in \H$ in $G'$;
\item[(iii)]  there exists no $G$-equivariant continuous map 
from $\partial(G,\H)$ to $\partial(G',\H')$;
\item[(iv)] the group $G$ is not quasiconvex relative to $\H'$ in $G'$. 
\end{enumerate}
\end{Thm}
\noindent 
If we apply Theorem \ref{noCT} for the case where $G$ is hyperbolic and $\H$ is $\emptyset$, 
then $G'$ is hyperbolic and $\H'$ is $\emptyset$. 
We remark that our proof uses \cite[Theorem 1]{B-R12}.

\begin{Rem}
Theorem \ref{noCT} (i), (ii) and (iv) can be considered 
as a generalization of \cite[Theorem A]{Kap99}
for relatively hyperbolic groups. 
\end{Rem}

Let $G$ be a countable group and $X$ be a compact metrizable space endowed 
with a minimal non-elementary convergence action of $G$. 
We denote by $\H(G,X)$ 
the set of all maximal parabolic subgroups with respect to the action of $G$ on $X$
and call it the peripheral structure with respect to the action of $G$ on $X$. 
Let us consider another compact metrizable space $Y$
endowed with a minimal non-elementary convergence action of $G$. 
When there exists a $G$-equivariant continuous map 
from $X$ to $Y$, we say that $X$ is a blow-up of $Y$ 
and that $Y$ is a blow-down of $X$.
Suppose that the action of $G$ on $X$ is geometrically finite.
\cite[Proposition 1.6]{M-O-Y5} claims that 
$X$ has no proper blow-ups with the same peripheral structure.
\cite[Theorem 1.4]{M-O-Y5} gives a family of uncountably infinitely many 
blow-downs of $X$ with the same peripheral structure.
On the other hand Theorem \ref{noCT} implies that 
there exists a compact metrizable space
endowed with a minimal non-elementary convergence action of $G$ 
such that the peripheral structure is equal to $\H(G,X)$ 
and it is not a blow-down of $X$.  
In fact the following is shown: 
\begin{Cor}\label{notdown}
Let $G$ be a countable group. 
Let $X$ be a compact metrizable space endowed with 
a geometrically finite convergence action of $G$. 
Then there exists a compact metrizable space $Y$ endowed with 
a minimal non-elementary convergence action of $G$
satisfying the following 
\begin{enumerate}
\item[(i)] $\H(G,X)=\H(G,Y)$;
\item[(ii)] the spaces $X$ and $Y$ has no common blow-ups. 
In particular $Y$ is not a blow-down of $X$.
\end{enumerate}
\end{Cor}

\begin{Rem}
For every non-elementary relatively hyperbolic group
(resp. every non-elementary hyperbolic group), 
the second question (resp. the first question) in \cite[Section 1]{Mj09}
has a negative answer by this corollary. 
Also this corollary implies that every non-elementary relatively hyperbolic group
does not have the universal convergence action which is defined 
by Gerasimov (\cite[Subsection 2.4]{Ger09}). 
\end{Rem}

\section{Proof of Theorem \ref{noCT}} 
Before we show Theorem \ref{noCT}, we fix some notations. 
Let a countable group $G$ act on a compact metrizable space $X$.
Suppose that the action is a minimal non-elementary convergence action.  
Then $X$ can be regarded as a boundary of $G$. 
In fact $G\cup X$ has the unique topology such that 
this is a compactification of $G$ and the natural action on $G\cup X$ 
is a convergence action whose limit set is $X$ 
(see for example \cite[Lemma 2.1]{M-O-Y5}). 
Let $L$ be a subgroup of $G$. 
Then the restricted action of $L$ on $X$
is a convergence group action. 
We denote by $\Lambda(L,X)$ the limit set. 
If $L$ is neither virtually cyclic nor 
parabolic with respect to the action on $X$, 
then the induced action of $L$ on $\Lambda(L,X)$ 
is also a minimal non-elementary convergence action. 

We need the following lemma in order to show Theorem \ref{noCT}. 
\begin{Lem}\label{extension}
Let $G'$ be a countable group and have a subgroup $G$. 
Let $X$ an $X'$ be compact metrizable spaces 
endowed with minimal non-elementary convergence actions of $G$ and $G'$, respectively.
Then the following is equivalent: 
\begin{enumerate}
\item[(i)] there exists a $G$-equivariant continuous map from $X$ to $X'$;
\item[(ii)] there exists a $G$-equivariant continuous map from $X$ to $\Lambda(G,X')$;
\item[(iii)] the injection $G\to G'$ is continuously extended to a map 
from $G\cup X$ to $G'\cup X'$.
\end{enumerate}
\end{Lem}
\begin{proof}
The implication from (iii) to (i) (resp. from (i) to (ii)) is trivial. 
We show that (ii) implies (iii). 
Suppose that we have a $G$-equivariant continuous map $\phi$ from $X$ to $\Lambda(G,X')$. 
Then this is extended to a continuous map 
$id_G\cup \phi:G\cup X\to G\cup \Lambda(G,X')$ (see \cite[Lemma 2.3]{M-O-Y5}).
Since $G\cup \Lambda(G,X')$ is regarded as 
the closure of $G$-orbit of the unit element of $G'$ in $G'\cup X'$,
the injection $\iota: G\cup \Lambda(G,X')\to G'\cup X'$ is continuous. 
Then $\iota\circ (id_G\cup \phi):G\cup X\to G'\cup X'$ 
is a continuous extension of the injection $G\to G'$.
\end{proof}

\begin{proof}[Proof of Theorem \ref{noCT}]
Since $G$ has the maximal finite normal subgroup by \cite[Lemma 3.3]{A-M-O07}, 
we denote it by $M(G)$. 
By using \cite[Theorem B.1]{M-O-Y5}, 
we take a subgroup $F'=F\times M(G)$ of $G$ 
such that $F$ is a free group of rank $3$ and 
$G$ is hyperbolic relative to 
\[
\H\cup \{K\subset G\ |\ K=gF'g^{-1} \text{ for some }g\in G\}.
\] 
Take a hyperbolic group $L$ containing $F$ as a subgroup
such that the injection $F\to L$ can not continuously extend 
on the Gromov-boundaries by \cite[Theorem 1]{B-R12}. 
We remark that 
there exists no $F$-equivariant continuous map from $\partial F$ to $\partial L$ 
by Lemma \ref{extension}. 
We put $L':=L\times M(G)$, $G':=G*_{F'}L'$ and 
\[
\H':=\{H'\subset G'\ |\ H'=g'Hg'^{-1} \text{ for some }H\in \H
\text{ and for some }g'\in G'\}.
\] 
By the construction, we have the condition (ii).
Also it follows from \cite[Theorem 0.1 (2)]{Dah03} 
that $G'$ is hyperbolic relative to 
\[
\H'\cup \{K\subset G'\ |\ K=g'L'g'^{-1} \text{ for some }g'\in G'\}.
\] 
Since $L'$ is hyperbolic, we have the condition (i) by \cite[Theorem 2.40]{Osi06a}. 

Now we show the condition (iii). 
Assume that there exists a $G$-equivariant continuous map $\phi:\partial(G,\H)\to \partial(G',\H')$.  
The map $\phi$ implies $F$-equivariant continuous map 
$\Lambda(F,\partial(G,\H))\to \Lambda(L,\partial(G',\H'))$. 
Since $F'$ (resp. $L'$) is hyperbolically embedded into $G$ (resp. $G'$) 
relative to $\H$ (resp. $\H'$) in the sense of \cite[Definition 1.4]{Osi06b}, 
$F$ (resp. $L$) is strongly quasiconvex relative to $\H$ (resp. $\H'$) in $G$ (resp. $G'$)
in the sense of \cite[Definition 4.11]{Osi06a} 
by \cite[Theorem 1.5]{Osi06b} and \cite[Theorem 4.13]{Osi06a}. 
Then the action on $F$ (resp. $L$) on $\Lambda(F,\partial(G,\H))$ 
(resp. $\Lambda(L,\partial(G',\H'))$) is a geometrically finite convergence action
without parabolic points (see \cite[Theorem 9.9]{Hru10}). 
Hence $\Lambda(F,\partial(G,\H))$ (resp. $\Lambda(L,\partial(G',\H'))$)
is $F$-equivariant (resp. $L$-equivariant) homeomorphic to 
the Gromov boundary $\partial F$ (resp. $\partial L$) 
by \cite[Theorem 0.1]{Bow98} and \cite[Theorem 1A]{Tuk98}. 
Hence $\phi$ gives an $F$-equivariant continuous map from $\partial F$ to $\partial L$. 
This contradicts the fact that there exists no such maps.

Finally we show the condition (iv).
Assume that $G$ is quasiconvex relative to $\H'$ in $G'$. 
The peripheral structure with respect to the action of $G$ on $\partial(G',\H')$ is  
\[\H'':=\{P\subset G \ |\ P\text{ is infinite and }P=G\cap H'\text{ for some }H'\in\H'\}.\]
By \cite[Definition 1.6]{Dah03}, 
$\partial (G,\H'')$ is $G$-equivariant homeomorphic to $\Lambda(G,\partial(G',\H'))$. 
Since we have $\H\subset \H''$ by the condition (ii), 
there exists a $G$-equivariant continuous map from 
$\partial (G,\H)$ to $\partial(G,\H'')$ by \cite[Theorem 1.1]{M-O-Y5}. 
Thus we have a $G$-equivariant continuous map from $\partial (G,\H)$ to $\partial (G',\H')$. 
This contradicts the condition (iii).
\end{proof}

\begin{proof}[Proof of Corollary \ref{notdown}]
For $G$ and $\H:=\H(G,X)$, 
we take $G'$ and $\H'$ in Theorem \ref{noCT} and put $Y:=\Lambda(G,\partial (G',\H'))$. 
Note that $X=\partial(G,\H)$ and $\H=\H(G,Y)$.
Assume that there exists a common blow-up of $X$ and $Y$. 
Since $\H(G,X)=\H(G,Y)$, we have a compact metrizable space $Z$ endowed with 
a minimal non-elementary convergence action of $G$ which 
is a common blow-up of $X$ and $Y$ such that $\H(G,Z)=\H(G,X)=\H(G,Y)$ 
by \cite[Lemma 2.6]{M-O-Y5}.  
Then the $G$-equivariant continuous map from $Z$ to $X$ is a
$G$-equivariant homeomorphism by \cite[Proposition 1.6]{M-O-Y5}. 
Hence $Y$ is a blow-down of $X$ and thus we have a $G$-equivariant continuous map 
from $\partial(G,\H)$ to $\partial(G',\H')$ by Lemma \ref{extension}. 
This contradicts the condition (iii) in Theorem \ref{noCT}. 
\end{proof}

\begin{Rem}
The space $Y$ in the above proof cannot be written as inverse limit of any inverse system of 
compact metrizable spaces endowed with geometrically finite convergence actions of $G$
(compare with \cite[Theorem 1.4]{M-O-Y5}). 
Indeed assume that $Y$ is inverse limit of an inverse system of 
compact metrizable spaces $X_i$ $(i\in I)$ 
endowed with geometrically finite convergence actions of $G$.
Since every element $H\in \H$ is parabolic with respect to the action on $Y$ 
and thus on $X_i$ for each $i\in I$, 
there exists a unique $G$-equivariant continuous map 
from $\partial (G,\H)$ to $X_i$ for each $i\in I$
by \cite[Theorem 1.1]{M-O-Y5}. 
Hence we have a $G$-equivariant continuous map from $\partial (G,\H)$ to $Y$. 
This contradicts Corollary \ref{notdown}. 

It may be interesting to ask whether a given 
compact metrizable space endowed with a geometrically infinite convergence action of $G$
can be written as inverse limit of some inverse system of 
compact metrizable spaces endowed with geometrically finite convergence actions of $G$. 
\end{Rem}

%%%%%%%%%%%%%%%%%%%%%%%%%%%%%%%%%%

\end{document}